\documentclass{article}

\usepackage{amsmath,amsthm,amssymb,amsfonts}
\usepackage{verbatim}
\usepackage{stmaryrd} 
\usepackage{hyperref}
\usepackage[colorinlistoftodos]{todonotes}
\usepackage{tikz}
\usepackage{tikz-cd}
\tikzset{vertex/.style={circle,draw,fill,inner sep=0pt,minimum size=1mm}}

\theoremstyle{plain}

\newtheorem{thm}{Theorem}
\newtheorem*{thm*}{Theorem}
\newtheorem*{lem*}{Lemma}
\newtheorem{lem}[thm]{Lemma}
\newtheorem{cor}[thm]{Corollary}
\newtheorem*{cor*}{Corollary}
\newtheorem{prop}[thm]{Proposition}
\newtheorem{conj}[thm]{Conjecture}
\newtheorem*{conj*}{Conjecture}

\theoremstyle{definition}
\newtheorem{exl}[thm]{Example}
\newtheorem*{exl*}{Example}

\newtheorem{defn}[thm]{Definition}

\newcommand{\R}{\mathbb{R}}
\newcommand{\Z}{\mathbb{Z}}

\newcommand{\adj}{\leftrightarrow}
\newcommand{\adjeq}{\leftrightarroweq}

\DeclareMathOperator{\id}{id}
\DeclareMathOperator{\im}{im}

\DeclareMathOperator{\NP}{NP}

\newcommand{\pd}{\partial}

\renewcommand{\v}{\check}

\begin{document}

%
%

\title{Digital homotopy relations and digital homology theories}
\author{P. Christopher Staecker}
\maketitle

\begin{abstract}
In this paper we prove results relating to two homotopy relations and four homology theories developed in the topology of digital images.

We introduce a new type of homotopy relation for digitally continuous functions which we call ``strong homotopy.'' Both digital homotopy and strong homotopy are natural digitizations of classical topological homotopy: the difference between them is analogous to the difference between digital 4-adjacency and 8-adjacency in the plane.

We also consider four different digital homology theories: a simplicial homology theory by Arslan et al which is the homology of the clique complex, a singular simplicial homology theory by D. W. Lee, a cubical homology theory by Jamil and Ali, and a new kind of cubical homology for digital images with $c_1$-adjacency which is easily computed, and generalizes a construction by Karaca \& Ege. We show that the two simplicial homology theories are isomorphic to each other, but distinct from the two cubical theories.

We also show that homotopic maps have the same induced homomorphisms in the cubical homology theory, and strong homotopic maps additionally have the same induced homomorphisms in the simplicial theory.
\end{abstract}

\section{Introduction}
A \emph{digital image} is a set of points $X$, with some \emph{adjacency relation} $\kappa$, which is symmetric and antireflexive. The standard notation for such a digital image is $(X, \kappa)$. Typically in digital topology the set $X$ is a finite subset of $\Z^n$, and the adjacency relation is based on some notion of adjacency of points in the integer lattice. This style of digital topology has its origins in the work of Rosenfeld and others, see \cite{rose86} for an early work.

We will use the notation $x \adj_\kappa y$ when $x$ is adjacent to $y$ by the adjacency $\kappa$, and $x \adjeq_\kappa y$ when $x$ is adjacent or equal to $y$. The particular adjacency relation will usually be clear from context, and in this case we will omit the subscript.

The results of this paper hold generally, without any specific reference to the embedding of the images in $\Z^n$. Thus we will often simply consider a digital image $X$ as an abstract simple graph, where the vertices are points of $X$, and an edge connects two vertices $x, x' \in X$ whenever $x\adj x'$. 

\begin{defn} 
Let $(X, \kappa), (Y, \lambda)$ be digital images. A function $f: X \rightarrow Y$ is $(\kappa, \lambda)$-continuous iff whenever $x \adj_{\kappa} y$ then $f(x) \adjeq_{\lambda} f(y)$. 
\end{defn}

When $f$ is a continuous bijection with continuous inverse, we say $f$ is an \emph{isomorphism}. 

For simplicity of notation, we will generally not need to reference the adjacency relation specifically. Thus we typically will denote a digital image simply by $X$, and when the appropriate adjacency relations are clear we simply call a function between digital images ``continuous''. We will often refer to digital images as simply ``images''.

For any digital image $X$ the identity map $\id_X:X \to X$ is always continuous.

The topological theory of digital images has, to a large part, been characterized by taking ideas from classical topology and ``discretizing'' them. Typically $\R^n$ is replaced by $\Z^n$, and so on. 
Viewing $\Z$ as a digital image, it makes sense to say $a$ and $b$ are adjacent if and only if $|a-b|=1$. This adjacency relation corresponds to connectivity in the standard topology of $\R$. 

When $n>1$ there is no canonical ``standard adjacency'' to use in $\Z^n$ which corresponds naturally to the standard topology of $\R^n$. In the case of $\Z^2$, for example, at least two different adjacency relations seem reasonable: we can view $\Z^2$ as a rectangular lattice connected by the coordinate grid, so that each point is adjacent to 4 neighbors; or we can additionally allow diagonal adjacencies so that each point is adjacent to 8 neighbors. This is formalized in the following definition from \cite{han05} (though these adjacencies had been studied for many years earlier):

\begin{defn}
Let $k,n$ be positive integers with $k\le n$. Then define an adjacency relation $c_k$ on $\Z^n$ as follows: two points $x,y\in \Z^n$ are $c_k$-adjacent if their coordinates differ by at most 1 in at most $k$ positions, and are equal in all other positions.
\end{defn}

We will also make use of the notion of connectedness. 
For $a,b\in \Z$  and $a<b$, let $[a,b]_\Z$ denote the set $\{a,a+1,\dots, b\}$. This set is called the \emph{digital interval} from $a$ to $b$. Given two points $x,y\in (X,\kappa)$, a \emph{$\kappa$-path from $x$ to $y$} is a $(c_1,\kappa)$-continuous function $p: [0,n]_\Z \to X$ with $p(0)=x$ and $p(n)=y$. When the adjacency relation is understood, a $\kappa$-path is called simply a \emph{path}. 

A digital image $(X,\kappa)$ is \emph{connected} when any two points of $X$ can be joined by a path. A \emph{connected component} of $X$ is a maximal connected subset of $X$.

The following is the standard definition of digital homotopy. For $a,b\in \Z$ with $a\le b$, let $[a,b]_\Z$ be the \emph{digital interval} $[a,b]_\Z = \{a, a+1, \dots, b\}$.

\begin{defn}\cite{boxe99}
\label{homotopydef}
Let $(X,\kappa)$ and $(Y,\lambda)$ be digital images.
Let $f,g: X \rightarrow Y$ be $(\kappa,\lambda)$-continuous functions.
Suppose there is a positive integer $m$ and a function
$H: X \times [0,k]_{\Z} \rightarrow Y$
such that:

\begin{itemize}
\item for all $x \in X$, $H(x,0) = f(x)$ and $H(x,k) = g(x)$;
\item for all $x \in X$, the induced function
      $H_x: [0,k]_{\Z} \rightarrow Y$ defined by
          \[ H_x(t) ~=~ H(x,t) \mbox{ for all } t \in [0,k]_{\Z} \]
          is $(c_1,\lambda)-$continuous. 
\item for all $t \in [0,k]_{\Z}$, the induced function
         $H_t: X \rightarrow Y$ defined by
          \[ H_t(x) ~=~ H(x,t) \mbox{ for all } x \in  X \]
          is $(\kappa,\lambda)-$continuous.
\end{itemize}
Then $H$ is a {\em [digital] homotopy between} $f$ and
$g$, and $f$ and $g$ are {\em [digitally] homotopic}, denoted
$f \simeq g$.

If $k=1$, then $f$ and $g$ are {\em homotopic in one step}.
\end{defn}

Homotopy in one step can easily be expressed in terms of individual adjacencies:
\begin{prop}\label{htpadjs}
Let $f,g:X\to Y$ be continuous. Then $f$ is homotopic to $g$ in one step if and only if for every $x\in X$, we have $f(x)\adjeq g(x)$.
\end{prop}
\begin{proof}
Assume that $f$ is homotopic to $g$ in one step. The homotopy is simply defined by $H(x,0)=f(x)$ and $H(x,1)=g(x)$. Then by continuity in the second coordinate of $H$, we have $f(x) = H(x,0) \adjeq H(x,1) = g(x)$ as desired.

Now assume that $f(x) \adjeq g(x)$ for each $x$. Define $H:X\times [0,1]_\Z \to Y$ by $H(x,0)=f(x)$ and $H(x,1)=g(x)$. Clearly $H(x,t)$ is continuous in $x$ for each fixed $t$, since $f$ and $g$ are continuous. Also $H(x,t)$ is continuous in $t$ for fixed $x$ because $H(x,0) = f(x) \adjeq g(x) = H(x,1)$. Thus $H$ is a one step homotopy from $f$ to $g$ as desired.
\end{proof}

Definition \ref{homotopydef} is inspired by the concept of homotopy from classical topology, 
but the classical definition is simpler because it can use the product topology. 
In classical topology the continuity of the two types of ``induced function'' is expressed simply by saying that $H:X\times [0,1] \to Y$ is continuous with respect to the product topology on $X \times [0,1]$.

Given two digital images $A$ and $B$, we can consider the product $A\times B$ as a digital image, but there are several choices for the adjacency to be used. The most natural adjacencies are the \emph{normal product adjacencies}, which were defined by Boxer in \cite{boxe17}. This generalizes an earlier product construction from \cite{han05}, which is equivalent to $\NP_2$.
\begin{defn}\cite{boxe17}
For each $i\in \{1,\dots,n\}$, let $(X_i,\kappa_i)$ be digital images. Then for some $u \in \{1,\dots,n\}$, the \emph{normal product adjacency} $\NP_u(\kappa_1,\dots,\kappa_n)$ is the adjacency relation on $\prod_{i=1}^n X_i$ defined by: $(x_1,\dots, x_n)$ and $(x'_1,\dots, x'_n)$ are adjacent if and only if their coordinates are adjacent in at most $u$ positions, and equal in all other positions.

When the underlying adjacencies are clear, we abbreviate $\NP_u(\kappa_1,\dots,\kappa_n)$ as simply $\NP_u$.
\end{defn}

The normal product adjacency is inspired by the various standard adjacencies typically used on $\Z^n$. On $\Z^1$, as mentioned above, the standard adjacency is $c_1$. Viewing $\Z^2$ as the product $\Z^2 = \Z \times \Z$, it is easy to see that 4-adjacency is the same as $\NP_1(c_1,c_1)$, and 8-adjacency is $\NP_2(c_1,c_1)$. The typical adjacencies used in $\Z^3$ are 6-, 18-, and 26-adjacency, depending on which types of diagonal adjacencies are allowed. These adjacencies are exactly $\NP_1(c_1,c_1,c_1)$, $\NP_2(c_1,c_1,c_1)$, and $\NP_3(c_1,c_1,c_1)$. 

Boxer showed that the definition of homotopy can be rephrased in terms of an $\NP_1$ product adjacency:
\begin{thm}\cite[Theorem 3.6]{boxe17}\label{np1thm}
Let $(X,\kappa)$ and $(Y,\lambda)$ be digital images. Then $H:X\times [0,k]_\Z \to Y$ is a homotopy if and only if $H$ is $(\NP_1(\kappa,c_1),\lambda)$-continuous.
\end{thm}

Once we have a notion of homotopy, it is natural to define homotopy equivalence:
\begin{defn}
Digital images $X$ and $Y$ are \emph{homotopy equivalent} when there exist continuous functions $f:X\to Y$ and $g:Y\to X$ with $g\circ f \simeq \id_X$ and $f\circ g \simeq \id_Y$, where $\id_X$ and $\id_Y$ denote the identity functions on $X$ and $Y$. In this case $f$ and $g$ are called \emph{homotopy equivalences}. 
\end{defn}

The structure of this paper is as follows: in Section \ref{strongsection} we define strong homotopy and give some of its basic properties. In Section \ref{punctuatedsec} we show that a homotopy is strong if and only if it can be made ``punctuated'', that is, changing by only one point at a time. In Section \ref{homologytheoriessection} we describe three different digital homology theories existing in the literature. In Section \ref{c1sec} we describe a new theory, defined only for images in $\Z^n$ with $c_1$-adjacency, which we call $c_1$-cubical homology, similar to a construction described in \cite{ke12}. In Section \ref{sxequivsection} we show that the simplicial and singular theories are isomorphic. In Section \ref{homologysec} we prove homotopy and strong homotopy invariance properties for the various theories. In Section \ref{relationshipsection} we describe some relationships between the simplicial and cubical theories, and in Section \ref{exlsection} we describe some specific examples.

The author would like to thank Samira Jamil and Danish Ali for helpful conversations concerning early drafts of this paper.

\section{Strong homotopy}\label{strongsection}

Theorem \ref{np1thm} states that a homotopy is a function $H:X\times [0,k]_\Z \to Y$ which is continuous when we use the $\NP_1$ product adjacency in the domain. We will explore the following definition, which simply uses $\NP_2$ in place of $\NP_1$. As we will see this imposes extra restrictions on the homotopy $H$.

\begin{defn}
Let $(X,\kappa)$ and $(Y,\lambda)$ be digital images. We say $H:X\times [0,k]_\Z \to Y$ is a \emph{strong homotopy} when $H$ is $(\NP_2(\kappa,c_1), \lambda)$-continuous.

If there is a strong homotopy $H$ between $f$ and $g$, we say $f$ and $g$ are \emph{strongly homotopic}, and we write $f \cong g$. If additionally $k=1$, we say $f$ and $g$ are \emph{strongly homotopic in one step}.
\end{defn}

Since we have exchanged $\NP_1$ for $\NP_2$ in the definition above, we may say informally that strong homotopy and digital homotopy provide two different but equally natural ``digitizations'' of the classical topological idea of homotopy. The difference between homotopy and strong homotopy of continuous functions is analogous to the difference between 4-adjacency and 8-adjacency of points in the plane. The strong homotopy relation matches one used in recent work \cite{los19} to build a digital homotopy theory in many respect matching the classical one.

It is clear from the definition of the normal product adjacency that if two points are $\NP_1$-adjacent, then they are $\NP_2$-adjacent. Thus any function $f:A\times B \to C$ which is continuous when using $\NP_2$ in the domain will automatically be continuous when using $\NP_1$ in the domain. Thus we obtain the following, which justifies the use of the word ``strong.''
\begin{thm}\label{strongimplieshtp}
Let $H:X\times [0,k]_\Z \to Y$ be a strong homotopy. Then $H$ is a homotopy.
\end{thm}

A standard argument shows that strong homotopy is an equivalence relation.
\begin{thm}
Strong homotopy is an equivalence relation.
\end{thm}

Strong homotopy in one step can be expressed in terms of adjacencies as in Proposition \ref{htpadjs}.
\begin{thm}\label{stronghtpadjs}
Let $f,g:X\to Y$ be continuous. Then $f$ is strongly homotopic to $g$ in one step if and only if for every $x,x'\in X$ with $x\adj x'$, we have $f(x)\adjeq g(x')$.
\end{thm}
\begin{proof}
First assume that $f$ is homotopic to $g$ in one step. The homotopy is simply defined by $H(x,0)=f(x)$ and $H(x,1)=g(x)$. Then if $x \adj x'$, we will have $(x,0) \adj_{\NP_2} (x',1)$, and thus since $H$ is $\NP_2$-continuous we have $f(x) = H(x,0) \adjeq H(x',1) = g(x')$ as desired.

Now assume that $f(x) \adjeq g(x')$ for each $x$. Define $H:X\times [0,1]_\Z \to Y$ by $H(x,0)=f(x)$ and $H(x,1)=g(x)$, and we must show that $H$ is $\NP_2$-continuous. Take $(x,t),(x',t') \in X\times [0,1]_\Z$ with $(x,t)\adj_{\NP_2} (x',t')$, and we will show that $H(x,t)\adjeq H(x',t')$. We have a few cases based on the values of $t,t'\in \{0,1\}$.

If $t=t'$, without loss of generality say $t=t'=0$. Since $(x,t)\adj_{\NP_2} (x',t')$, we have $x\adjeq x'$, and so we have 
\[ H(x,t) = H(x,0) = f(x) \adjeq f(x') = H(x',0) = H(x',t') \]
since $f$ is continuous. Thus $H(x,t)\adjeq H(x',t')$ as desired.

Finally, if $t\neq t'$, without loss of generality assume $t=0$ and $t'=1$. Since $(x,t)\adj_{\NP_2} (x',t')$, we have $x\adjeq x'$, and so we have 
\[ H(x,t) = H(x,0) = f(x) \adjeq g(x') = H(x',1) = H(x',t') \]
and so $H(x,t)\adjeq H(x',t')$ as desired.
\end{proof}

By Theorem \ref{strongimplieshtp}, if $f \cong g$ then automatically we have $f \simeq g$. The converse is not true, however, as the following example shows.

One important source of examples in the study of digital images is the \emph{digital cycle} $C_n = \{c_0,\dots, c_{n-1}\}$, with adjacency given by $c_i \adj c_{i+1}$ for each $i$, where for convenience we always read the subscripts modulo $n$. Thus $C_n$ is a digital image of $n$ points which is in many ways topologically analogous to the circle.

\begin{exl}\label{strongrigid}
It is well known that all selfmaps on $C_4$ are homotopic to one another. But we will show that the identity map $\id_{C_4}:C_4 \to C_4$ is not strongly homotopic to any map $f$ with $\#f(C_4) < 4$. It will suffice to show that $\id_{C_4}$ is not strongly homotopic in 1 step to any such map. 

Without loss of generality assume that $f(C_4) \subseteq \{c_0,c_1,c_2\}$, and assume for the sake of a contradiction that $f$ is strongly homotopic in one step to $\id_{C_4}$. Then by Theorem \ref{stronghtpadjs}, since $c_2 \adj c_3$ and $c_0\adj c_3$, we will have $c_2 \adjeq f(c_3)$ and also $c_0 \adjeq f(c_3)$. Thus $f(c_3)$ is adjacent to both $c_0$ and $c_2$, but cannot equal $c_3$ since $c_3 \not\in f(C_4)$. We conclude that $f(c_3)=c_1$. By Theorem \ref{htpadjs}, this contradicts the fact that $f$ is homotopic to the identity in 1 step.
\end{exl}

\section{Punctuated homotopy}\label{punctuatedsec}
Given a homotopy $H(x,t)$, we say that $H$ is \emph{punctuated} if, for each $t$, there is some $x_t$ such that $H(x,t) = H(x,t+1)$ for all $x\neq x_t$. That is, from each stage of the homotopy to the next, the induced map $H_t(x)$ is changing by at most one point at a time. 

\begin{thm}
Any punctuated homotopy is a strong homotopy.
\end{thm}
\begin{proof}
Let $H$ be a punctuated homotopy, and let $(x,t) \adj_{\NP_2} (x',t')$. We must show that $H(x,t) \adjeq H(x',t')$. Since $(x,t) \adj_{\NP_2} (x',t')$ we have $x \adjeq x'$ and $t \adjeq t'$. If $t=t'$, then we have $(x,t) \adj_{\NP_1} (x',t')$ and so $H(x,t) \adjeq H(x',t')$ since $H$ is a homotopy. It remains to consider when $t \adj t'$ and $t\neq t'$. In this case, without loss of generality assume $t'=t+1$.

Since $H$ is a punctuated homotopy, there is some $x_t$ such that $H(x,t) = H(x,t+1)$ for all $x\neq x_t$. 
When $x \neq x_t$, we have 
\[ H(x,t) = H(x,t+1) \adjeq H(x',t+1) = H(x',t') \]
since $H$ is a homotopy. Similarly we have $H(x,t) \adjeq H(x',t')$ when $x' \neq x_t$.

Thus it remains only to consider when $x=x'=x_t$. That is, we must show that $H(x_t,t) \adjeq H(x_t,t') = H(x_t,t+1)$, and this is true because $H$ is a homotopy.
\end{proof}

We also have a sort of converse to the above. While not every strong homotopy is punctuated, any two strongly homotopic maps can be connected by a punctuated homotopy, provided that the domain is finite.

\begin{thm}\label{punctuatedthm}
Let $X$ be a finite digital image, and let $f,g:X\to Y$ be strongly homotopic. Then $f$ and $g$ are homotopic by a punctuated homotopy.
\end{thm}
\begin{proof}
By induction, it suffices to show that if $f$ and $g$ are strongly homotopic in one step, then they are homotopic by a punctuated homotopy. Enumerate the points of $X$ as $X = \{x_0,\dots, x_n\}$, and define $H:X\times [0,n+1]_\Z \to Y$ by:
\[ H(x_i,t) = \begin{cases} f(x_i)  & \quad\text{ if } i \ge t, \\
g(x_i) &\quad\text{ if } i < t. \end{cases} \]
Then $H$ moves one point at a time, so we need only to show that it has the appropriate continuity properties to be a homotopy. 

First we show that $H(x,t)$ is continuous in $x$ for fixed $t$. Take $x \adj x'$, then $H(x,t) \in \{f(x),g(x)\}$ and $H(x',t) \in \{f(x'),g(x')\}$. Since $f$ and $g$ are homotopic in one step we have $f(x)\adjeq f(x')$ and $g(x) \adjeq g(x')$, and since $f$ and $g$ are strongly homotopic in one step we have $f(x) \adjeq g(x')$ and $g(x) \adjeq f(x')$. Thus in any case we have $H(x,t) \adjeq H(x',t)$ as desired.

Now we show that $H(x,t)$ is continuous in $t$ for fixed $x$. It suffices to show that $H(x,t) \adjeq H(x,t+1)$ for any $t$. We have $H(x,t) \in \{f(x),g(x)\}$ and also $H(x,t+1) \in \{f(x),g(x)\}$. Since $f$ and $g$ are homotopic in one step we have $f(x)\adjeq g(x)$, and thus $H(x,t) \adjeq H(x,t+1)$ as desired.
\end{proof}

The finiteness assumption above is necessary, as the following example shows.
\begin{exl}
Let $X = \Z \times \{0,1\} \subset \Z^2$, with 8-adjacency, and let $f(x,y) = (x,0)$. Then $f$ is strongly homotopic to $\id_X$ in one step. But $f(x,y)$ and $\id(x,y)$ differ for infinitely many values of $(x,y)\in X$. Since a punctuated homotopy has finite time interval, and can only change one value at a time, there can be no punctuated homotopy from $f$ to $\id_X$. 
\end{exl}

Combining the two theorems above gives us a nice characterization of strong homotopy. 
\begin{cor}
If $X$ is finite, two continuous maps $f,g:X\to Y$ are strongly homotopic if and only if they are homotopic by a punctuated homotopy.
\end{cor}

As an application, we show that the identity map on the $n$-cycle for $n\ge 4$ is not strongly homotopic to any other map:
\begin{exl}
Let $n \ge 4$, and assume for the sake of a contradiction that there is some continuous $f:C_n \to C_n$ with $f \cong \id_{C_n}$ and $f\neq \id_{C_n}$. Without loss of generality, assume that the homotopy from $f$ to $\id_{C_n}$ is punctuated and in one step. Since the homotopy is punctuated, $f$ moves one point, so without loss of generality we may assume that $f(c_0)=c_1$ and $f(c_i)=c_i$ for all $i\neq 0$. This contradicts the continuity of $f$, however, since we will have $c_0\adj c_{n-1}$ but $c_1 = f(c_0) \not \adjeq f(c_{n-1}) = c_{n-1}$ since $n\ge 4$. 
\end{exl}

\section{Digital homology theories in the literature}\label{homologytheoriessection}
Digital topological invariants are typically inspired by classical topology, though most are not literally topological in nature. For example, when $X$ and $Y$ are digital images, a digitally continuous function $f:X\to Y$ is not actually continuous in the classical sense with respect to any topologies on $X$ and $Y$. The digital fundamental group $\pi_1(X)$ is not actually the fundamental group of $X$ with respect to any topology on $X$, etc.

Digital homology, however, does fit neatly into the classical theory of homological algebra. Each of the homology theories described in \cite{bko11, lee14, ja19} is indeed a homology theory of a classical chain complex. Though it is not always done in these references, we will make use of results from classical homological algebra whenever possible to avoid the need for specialized definitions and proofs of basic results.

We will review the basic homological algebra that will be useful, see e.g. \cite{hatc02}.

A \emph{chain complex} is a sequence of abelian groups $C_0, C_1, \dots$ and homomorphisms $\pd_q: C_q \to C_{q-1}$ satisfying $\pd_{q-1} \circ \pd_q = 0$ for all $q$. Given some chain complex $C = (C_q,\pd_q)$, for each $q$ we define the \emph{cycle} and \emph{boundary} subgroups $Z_q$ and $B_q$ of $C_q$ as: $Z_q = \ker \pd_q$ and $B_q = \im \pd_{p+1}$. The \emph{dimension $q$ homology group} of the chain complex is defined as $H_q = Z_q / B_q$. When we are discussing the homology groups of various different chain complexes, we will write $H_q(C)$ for the dimension $q$ homology of the chain complex $C$.

Given two chain complexes $A = (A_q,\pd^A_q)$ and $B = (B_q,\pd^B_q)$, a sequence of homomorphisms $f_q:A_q \to B_q$ is a \emph{chain map from $A$ to $B$} when $f_{q-1} \circ \pd_q^A = \pd_q^B \circ f_q$ for each $q$. Every such chain map induces a well-defined sequence of homomorphisms $f_{*,q}: H_q(A) \to H_q(B)$. Furthermore, this correspondence is functorial in the sense that $(f\circ g)_{*,q} = f_{*,q} \circ g_{*,q}$, and the induced homomorphism of the identity function is the identity homomorphism. When the dimensions are clear, we will omit the subscript $q$.

Given three chain complexes $A$, $B$, $C$, and an exact sequence of chain maps:
\[ 0 \to A \xrightarrow{f} B \xrightarrow{g} C \to 0, \]
for each $q$ there is a \emph{connecting homomorphism} $\delta_q:H_q(C) \to H_{q-1}(A)$ such that the following sequence is exact:
\begin{equation}\label{longexact}
\dots \to H_{q+1}(C) \xrightarrow{\delta_{q+1}} H_q(A) \xrightarrow{f_{*,q}} H_q(B) \xrightarrow{g_{*,q}} H_q(C) \xrightarrow{\delta_q} H_{q-1}(A) \to \dots 
\end{equation}
The ``long exact sequence'' above is the fundamental tool of relative homology theory.

\subsection{Simplicial homology}
Digital simplicial homology theory was first defined by Arslan, Karaca, and \"Oztel in \cite{ako08}, in Turkish. This material was extended and published in English in \cite{bko11}. We review the definitions as presented in \cite{bko11}.

For a digital image $X$ and some $q\ge 0$, a $q$-simplex is defined to be any set of $q+1$ mutually adjacent points of $X$. For some ordered list of mutually adjacent points $x_0,\dots,x_q$, the associated ordered $q$-simplex is denoted $\langle x_0,\dots, x_q\rangle$. 

The \emph{chain group} $C_q(X)$ is defined to be the abelian group generated by the set of all ordered $q$-simplices, where if $\rho:\{0,\dots,q\} \to \{0,\dots,q\}$ is a permutation, then in $C_q(X)$ we identify
\begin{equation}\label{permeq}
\langle x_0,\dots, x_q \rangle = (-1)^\rho \langle x_{\rho(0)},\dots,x_{\rho(q)}\rangle, 
\end{equation}
where $(-1)^\rho = 1$ when $\rho$ is an even permutation, and $(-1)^\rho=-1$ when $\rho$ is an odd permutation.

The \emph{boundary homomorphism} $\pd_q: C_q(X) \to C_{q-1}(X)$ is the homomorphism induced by defining:
\begin{equation}\label{simplicialpd}
\pd_q(\langle x_0,\dots,x_q\rangle) = \sum_{i=0}^q (-1)^i \langle x_0,\dots,\widehat{x_i},\dots x_q\rangle,
\end{equation}
where $\widehat{x_i}$ indicates omission of the $x_i$ coordinate.

It can be verified that $\pd_{q-1} \circ \pd_q = 0$, and so $(C_q(X),\pd_q)$ forms a chain complex, and the \emph{dimension $q$ homology group} $H_q(X) = Z_q(X)/B_q(X)$ is defined to be the homology of this chain complex.

Any continuous function $f:X\to Y$ induces a homomorphism $f_{\#,q}:C_q(X) \to C_q(Y)$ defined by
\[ f_{\#,q}(\langle x_0,\dots, x_q\rangle) = \langle f(x_0),\dots,f(x_q)\rangle, \]
where the right side is interpreted as 0 if the set $\{f(x_0),\dots,f(x_q)\}$ has cardinality less than $q$. When the value of $q$ is understood, we simply write $f_{\#,q} = f_\#$. 

It is easy to check that this $f_{\#,q}$ is a chain map, and thus induces a homomorphism $f_{*,q}:H_q(X) \to H_q(X)$. Again, we typically write $f_{*,q} = f_*$ when the $q$ is understood.

We will remark that the constructions above match exactly the homology of the \emph{clique complex} of $X$ when viewed as a graph. The clique complex is the simplicial complex built from the complete subgraphs of a given graph, and the homology of this simplicial complex is the same as the digital homology defined above. The free mathematics software SageMath has built-in functions to compute the clique complex of a graph, and further to compute the homology of any finite complex. Thus it is easy to implement algorithms to compute the simplicial homology groups of a digital image. Source code for computing simplicial homology groups is available at the author's website for experimentation.\footnote{\url{http://faculty.fairfield.edu/cstaecker}}

These definitions and results are all exactly as expected from the classical homology theory of a simplicial complex. As an example, we compute the homology groups of the cycle $C_n$ for $n\ge 4$. The case $n=4$ appears as Theorem 3.17 of \cite{bko11}.
\begin{thm}\label{Cnhomology}
If $n\ge 4$, we have:
\[ H_q(C_n) = \begin{cases} \Z \quad &\text{ if $q\in \{0,1\}$,} \\
0 &\text{ if $q>1$.} \end{cases} \]
\end{thm}
\begin{proof}
First we prove the case $q=0$. The chain group $C_0(C_n)$ is generated by $n$ different $0$-simplices $\langle c_0\rangle, \dots, \langle c_{n-1}\rangle$. Since $\pd_0$ is a trivial homomorphism, we have $Z_0(C_n) = C_0(C_n)$. Note that for each $i$, we have 
\[ \langle c_i\rangle = (\langle c_i \rangle - \langle c_{i+1} \rangle) + \langle c_{i+1} \rangle = \pd \langle c_{i+1}, c_{i}\rangle + \langle c_{i+1} \rangle, \]
and thus $\langle c_i\rangle - \langle c_{i+1}\rangle \in B_0(C_n)$. Thus $\langle c_i \rangle$ and $\langle c_{i+1}\rangle$ are equal in $H_0(C_n)$ for every $i$. That is, $H_0(C_n)$ is the group generated by $\langle c_0\rangle$, and so $H_0(C_n) = \Z$ as desired.

Now for $q=1$, first we note that there are no $2$-simplices in $C_n$ (because $n\ge 4$), so $C_2(C_n)$ is trivial, and thus $B_1(C_n)$ is trivial. Thus $H_1(C_n)$ will be isomorphic to $Z_1(C_n)$. To determine $Z_1(C_n)$, we must determine which $\alpha \in C_1(C_n)$ satisfy $\pd \alpha = 0$. Any $\alpha\in C_1(C_n)$ can be expressed as:
\[ \alpha = w_1\langle c_0,c_1\rangle + \dots + w_n\langle c_{n-1}, c_0\rangle \]
for $w_i \in \Z$, and then $\pd \alpha = 0$ if and only if:
\begin{align*}
0 &= \pd \alpha = w_1(\langle c_1\rangle - \langle c_0\rangle) + \dots + w_n(\langle c_{0}\rangle - \langle c_{n-1}\rangle) \\
&= (w_n-w_1) \langle c_0 \rangle + (w_1-w_2)\langle c_1 \rangle + \dots + (w_{n-1}-w_n)\langle c_{n-1} \rangle
\end{align*}
and thus we have $w_1=w_2=\dots=w_n$ since the $\langle c_i\rangle$ are linearly independent in $C_1(C_n)$. Then we have shown that $H_1(C_n) = Z_1(C_n)$ is generated by the single element 
\[ \sigma = \sum_{i=0}^{n-1} \langle c_i,c_{i+1}\rangle, \]
and thus $H_1(C_n) = \Z$. 

For $q>1$, there are no $q$-simplices and so $C_q(C_n)$ is trivial, and thus $H_q(C_n)$ is trivial.
\end{proof}

In the case of $C_n$, we can make a full computation of the induced homomorphisms for any selfmap, using results from \cite{bs19}. Let $c:C_n \to C_n$ be the constant map $c(c_i) = c_0$, let $l:C_n \to C_n$ be the ``flip map'' $l(c_i) = c_{-i}$, and for some integer $d$, let $r_d:C_n \to C_n$ be the rotation $r_d(c_i) = c_{i+d}$. Theorem 9.3 of \cite{bs19} shows that there are exactly 3 homotopy classes of selfmaps on $C_n$: any map $f:C_n \to C_n$ is either strongly homotopic to a constant, or is homotopic to the identity and equals $r_d$ for some $d$, or is homotopic to the flip map and equals $r_d \circ l$ for some $d$. (The strongness of the homotopy to the constant was not mentioned in \cite{bs19}, but the homotopy demonstrated in the proof in that paper is easily made punctuated and therefore strong.)

\begin{thm}\label{Cninduced}
Let $n>4$, and let $f:C_n \to C_n$ be continuous. Then for all $q>1$, the induced homomorphism $f_{*,q}:H_q(C_n) \to H_q(C_n)$ is trivial, for $q=0$ the induced homomorphism $f_{*,0} = \id$, and for $q=1$ we have:
\[ f_{*,1} = \begin{cases} \id & \text{ if $f \simeq \id_{C_n},$} \\
-\id &\text{ if $f\simeq l$,} \\
0 &\text{ if $f \simeq c$.} \end{cases} \]
\end{thm}
\begin{proof}
For $q>1$ we have already seen that $H_q(C_n)$ is a trivial group, so we will have $f_*=0$ automatically.
When $f$ is homotopic to a constant, as mentioned above, in fact $f \cong c$ and thus $f_*=c_*$ for all $q$, and so $f_{*,0}$ is the identity and $f_{*,q}$ is trivial for $q>0$.

Now we consider when $f$ is homotopic to the identity or the flip map, for $q\in \{0,1\}$. First we consider $q=0$ and $f\simeq \id_{C_n}$. Since $H_0(C_n)$ is generated by $\langle c_0 \rangle$, it suffices to show that $f_*(\langle c_0\rangle) = \langle c_0 \rangle$. Let $d$ be some integer with $f= r_d$, and we have:
\[ f_*(\langle c_0 \rangle) = \langle f(c_0) \rangle = \langle c_d \rangle = \langle c_0 \rangle \]
as desired. Exactly the same argument applies for $f\simeq l$, since we will still have $\langle f(c_0) \rangle = \langle c_d \rangle$ for some $d$.

Now for $q=1$, and $f\simeq \id_{C_n}$ we must show $f_*(\sigma) = \sigma$, where $\sigma = \sum_{i=0}^{n-1} \langle c_i,c_{i+1} \rangle$ as in the proof of Theorem \ref{Cnhomology}. Let $f = r_d$, and we have:
\[ f(\sigma) =  \sum_{i=0}^{n-1} \langle f(c_i), f(c_{i+1}) \rangle = \sum_{i=0}^{n-1} \langle c_{i+d},c_{i+1+d} \rangle = \sigma \]
as desired.

Finally we consider $q=1$ and $f\simeq l$, and we must show $f_*(\sigma) = -\sigma$. Let $f=r_d \circ l$, so $f(c_i) = c_{d-i}$, and we have:
\begin{align*} 
f(\sigma) &= \sum_{i=0}^{n-1} \langle f(c_i), f(c_{i+1}) \rangle =  \sum_{i=0}^{n-1} \langle c_{d-i},c_{d-i-1} \rangle = \sum_{i=0}^{n-1} \langle c_{-i+1},c_{-i} \rangle  \\
&= \sum_{i=0}^{n-1} \langle c_{i+1},c_{i} \rangle = \sum_{i=0}^{n-1} -\langle c_i,c_{i+1}\rangle = -\sigma
\end{align*}
as desired.
\end{proof}

The three cases of Theorem \ref{Cninduced} suffice to compute $f_*$ for any selfmap of $C_n$, and we note that in each of the three homotopy classes, the set of induced homomorphisms is different. We obtain a sort of Hopf theorem for digital cycles:
\begin{cor}
Let $n> 4$, and let $f,g:C_n \to C_n$ be continuous. Then $f_{*,q} = g_{*,q}$ for each $q$ if and only if $f\simeq g$. 
\end{cor}

One major difference between the digital theory and classical homology is that the induced homomorphism $f_*$ is not always a digital homotopy invariant.

\begin{exl}
By Theorem \ref{Cnhomology}, the homology group $H_1(C_4)$ is isomorphic to $\Z$.
Because all maps on $C_4$ are homotopic, the identity map is homotopic to a constant map $c$. But $\id_*:H_1(C_4) \to H_1(C_4)$ is the identity homomorphism of $\Z$, while $c_*:H_1(C_4) \to H_1(C_4)$ is the trivial homomorphism. Thus $\id \simeq c$ but $\id_* \neq c_*$.
\end{exl}

The lack of a homotopy-invariant induced homorphism is a major deficiency in the homology theory of digital images. Lacking this homotopy invariance, the homology groups are not well-behaved with respect to typical topological constructions. For example two homotopy equivalent digital images may have different homology groups. As a consequence the digital Euler characteristic is not a digital homotopy type invariant.
Example 7.2 of \cite{ek13} shows that the Hurewicz theorem also fails: that is, that $H_1(X)$ may not be isomorphic to the abelianization of the fundamental group of $X$ as defined in \cite{boxe99}. 


The homology group in dimension zero is easy to predict:
\begin{thm}\label{simplicialdim0}
Let $X$ be any digital image with $d$ connected components. Then $H_0(X) \cong \Z^d$.
\end{thm}
\begin{proof}
The chain group $C_0(X)$ (which equals the group $Z_0(X)$ of 0-cycles) has basis given by the points of $X$. It is easy to see that two points $x,y\in Z_0(X)$ are homologous if and only if $x$ and $y$ are in the same connected component of $X$.
\end{proof}

The simplicial homology is also easy to compute when $X$ consists of finitely many isolated points, that is, points which are not adjacent to any other points.
\begin{thm}\label{simplicialisolated}
Let $X$ be any digital image consisting of $k$ isolated points. Then $H_0(X) \cong \Z^k$ and $H_q(X) = 0$ for $q>0$.
\end{thm}
\begin{proof}
The statement concerning $H_0(X)$ follows immediately from \ref{simplicialdim0}. Since $X$ has no adjacencies, it contains no $q$-simplices when $q>0$. Thus $H_q(X) = 0$ when $q>0$ as desired.
\end{proof}

\subsection{Singular homology}
Singular homology for digital images was defined by D.W. Lee in \cite{lee14}. We will review Lee's definitions, using some different notations to fit more cleanly with the other homology theories.

For any natural number $q$, let $\Delta^q$ be \emph{the standard $q$-simplex}, the digital image consisting of $q+1$ mutually adjacent points. Viewed as a graph, $\Delta^q$ is the complete graph of $q+1$ vertices. The points of $\Delta^q$ will be labeled and ordered as $\Delta^q = (e_0,\dots,e_q)$. 

\begin{defn}
Let $X$ be a digital image. A \emph{singular $q$-simplex} in $X$ is a continuous function
\[ \phi:\Delta^q \to X. \]
We will write such a singular $q$-simplex as the ordered list $[\phi(e_0),\dots, \phi(e_q)]$.

For any $q\ge 0$, the \emph{group of singular $q$-chains}, denoted $\v C_q(X)$, is the free abelian group whose basis is the set of all singular $q$-simplices of $X$. 
\end{defn}

The singular boundary operator $\pd_q:\v C_q(X) \to \v C_{q-1}(X)$ is defined as follows: 
\begin{equation}\label{singularboundary}
\pd_q [x_0,\dots,x_q] = \sum_{i=0}^q (-1)^i [x_0,\dots,\widehat{x_i},\dots, x_q], 
\end{equation}
where as usual $\widehat{x_i}$ denotes omission of the $i$th element. When $\phi$ is a $q$-simplex, $\pd_q\phi$ will be a singular $(q-1)$-chain. As before, we will often omit the subscript $q$. Note that we are using the same notation to denote the boundary operators in both simplicial and singular homology. In practice this will not cause confusion.

The singular $q$-simplex $[x_0,\dots,x_q]$ is very similar to the ordered $q$-simplex $\langle x_0,\dots,x_q\rangle$. The main difference algebraically is that, in the singular chain group, we do not identify permutations of the listings. For example we have $\langle x_0,x_1 \rangle = -\langle x_1,x_0\rangle$ in $C_1(X)$, but in $\v C_1(X)$ the basis elements $[x_0,x_1]$ and $[x_1,x_0]$ are linearly independent. We also will always have $\langle x,x\rangle = - \langle x,x\rangle = 0$ in $C_1(X)$, while $[x,x]$ will be nontrivial in $\v C_1(X)$.
This means that in particular $\v C_q(X)$ has nonzero elements which, when written as lists of points, include repetitions. Such elements are always 0 in $C_q(X)$ because of \eqref{permeq}.

Theorem 3.9 of \cite{lee14} shows that $\pd_{q-1} \circ \pd_{q} = 0$, and thus $(\v C_q(X),\pd_q)$ is a chain complex, and the \emph{dimension $q$ singular homology group} is defined as $\v H_q(X) = \v Z_q(X)/\v B_q(X)$.

If $f:X\to Y$ is continuous, then there is a homomorphism $f_{\#}:\v C_q(X) \to \v C_q(Y)$ defined on singular chains by $f_{\#}(\phi) = f \circ \phi$. This is easily shown to be a chain map, and thus we obtain the induced homomorphism on singular homology $f_{*}:\v H_q(X) \to \v H_q(Y)$.


We will require an analogue of Theorem \ref{simplicialdim0} for singular homology. The proof is the same as that of Theorem \ref{simplicialdim0}.
\begin{thm}\label{singulardim0}
Let $X$ be any digital image with $d$ connected components. Then $\v H_0(X) \cong \Z^d$.
\end{thm}

Lee's work provides an analogue of Theorem \ref{simplicialisolated} for singular homology. The following is a consequence of  Theorems 3.16 and 3.20 of \cite{lee14}:
\begin{thm}\label{singularisolated}
Let $X$ be a digital image which consists of $k$ isolated points. Then $\v H_0(X) \cong \Z^k$ and $\v H_q(X) = 0$ for $q>0$.
\end{thm}

\subsection{Cubical homology}
Cubical homology was introduced by Jamil \& Ali in \cite{ja19}, with definitions mimicking the classical cubical homology as presented in \cite{mass91}. We will review the definition and basic results.

Let $I = [0,1]_\Z = \{0,1\}$, and we consider $I^n \subset \Z^n$ as a digital image with $c_1$-adjacency for each $n\ge 1$. When $n=0$, we define $I^0$ to be a single point. 

\begin{defn}
Let $(X,\kappa)$ be a digital image. Then a \emph{$q$-cube} in $X$ is a $(c_1,\kappa)$-continuous function
\[ \sigma:I^q \to X. \]
A $q$-cube is \emph{degenerate} if there is some $i$ such that the function $\sigma(t_1,\dots,t_q)$ does not depend on the coordinate $t_i$. Let $Q_q(X)$ be the free abelian group whose basis is the set of all $q$-cubes in $X$, and let $D_q(X)$ be the subgroup generated by the degenerate $q$-cubes. Then the \emph{group of cubical $q$-chains}, denoted $\bar C_q(X)$, is the quotient $\bar C_q(X) = Q_q(X) / D_q(X)$.
\end{defn}

The  cubical boundary operator is defined in terms of cubical \emph{face} operators. For some $q$-cube $\sigma$ and some $i \in \{1,\dots,q\}$, define ($q-1$)-cubes $A_i\sigma$ and $B_i\sigma$ as:
\begin{align*}
(A_i\sigma)(t_1,\dots,t_{q-1}) &= \sigma(t_1,\dots, t_{i-1},0,t_i,\dots,t_{q-1}), \\
(B_i\sigma)(t_1,\dots,t_{q-1}) &= \sigma(t_1,\dots, t_{i-1},1,t_i,\dots,t_{q-1}), 
\end{align*}
These $A_i$ and $B_i$ give the ``front face'' and ``back face'' of the cube in each of its $q$ dimensions. 

The boundary operator $\pd_q:\bar C_q(X) \to \bar C_{q-1}(X)$ is defined on cubes by the formula:
\begin{equation}\label{cubepd}
\pd_q(\sigma) = \sum_{i=1}^q (-1)^i (A_i\sigma - B_i\sigma), 
\end{equation}
and extended to chains by linearity. We will sometimes omit the subscript $q$.

A routine calculation shows that $\pd_{q-1}\circ \pd_{q} = 0$, and thus $(\bar C_q(X), \pd_q)$ is a chain complex. We obtain cycle and boundary groups $\bar Z_q(X)$ and $\bar B_q(X)$ and the homology groups $\bar H_q(X) = \bar Z_q(X)/\bar B_q(X)$.

Exactly as in the singular theory, if $f:X\to Y$ is continuous, then $f_{\#}:\bar C_q(X) \to \bar C_q(Y)$ is defined on cubical chains by $f_{\#}(\sigma) = f \circ \sigma$, and this defines the induced homomorphism on cubical  homology $f_{*}:\bar H_q(X) \to \bar H_q(Y)$.

The most important property of cubical homology which distinguishes it from simplicial homology is the following, which is Theorem 3.7 of \cite{ja19}:
\begin{thm}[{\cite[Theorem 3.7]{ja19}}]
Let $X$ and $Y$ be any digital images and $f,g:X \to Y$ with induced homomorphisms $f_*,g_*:\bar H_q(X) \to \bar H_q(X)$. If $f\simeq g$, then $f_* = g_*$.
\end{thm}

Immediate corollaries include:
\begin{cor}[{\cite[Corollary 3.8]{ja19}}]\label{htpequivcube}
If $X$ and $Y$ are homotopy equivalent, then $\bar H_q(X) \cong \bar H_q(Y)$ for all $q$.
\end{cor}

\begin{cor}[{\cite[Example 3.9]{ja19}}]\label{contractiblecube}
If $X$ is contractible (i.e. $X$ is homotopy equivalent to a point), then 
\[ \bar H_q(X) = \begin{cases} \Z &\text{ if } q=0, \\
0 &\text{ if } q\neq 0.\end{cases} \]
\end{cor}

Jamil \& Ali also prove a Hurewicz theorem, that $\bar H_1(X)$ is isomorphic to the abelianization of the fundamental group of $X$. 
They also prove results concerning connected components and single points. The following theorems follow from \cite[Propositions 3.1, 3.2]{ja19} and Corollary \ref{contractiblecube}.
\begin{thm}\label{cubicaldim0}
Let $X$ be a digital image with $d$ connected components. Then $\bar H_0(X) \cong \Z^d$.
\end{thm}
\begin{thm}\label{cubicalisolated}
Let $X$ be a digital image consisting of $k$ isolated points. Then $\bar H_0(X) \cong \Z^k$ and $\bar H_q(X) \cong 0$ for $q>0$.
\end{thm}

\section{A new cubical homology theory for images with $c_1$-adjacency}\label{c1sec}
Ege \& Karaca in \cite{ke12} describe another type of cubical homology theory based on classical constructions from \cite{kmm03}. Their construction is not generally well-defined for any digital image, but only a digital image which is a ``cubical set'', that is, a finite union of ``elementary cubes.'' We will describe the construction, sometimes using different terminology that is more convenient for making comparisons with the other theories in this paper.

Ege \& Karaca's focus on ``cubical sets'' requires that the digital image $X$ be a subset of $\Z^n$, and that we always use $c_1$ as the adjacency relation. This is a significant restriction, but still allows many useful examples and results.

All of the following definitions appear in \cite{ke12}: an \emph{elementary interval} is a set of the form $[a,a+1]_\Z = \{a, a+1\}$ or $[a,a]_\Z = \{a\}$. An elementary interval of 1 point is called \emph{degenerate}, and one of 2 points is called \emph{nondegenerate}. An \emph{elementary cube} is any set:
\[ Q = J_1\times \dots \times J_q \subset \Z^q \]
where each $J_i$ is an elementary interval. The \emph{dimension} of $Q$ is the number of nondegenerate factors. An elementary cube of dimension $q$ will be called an \emph{elementary $q$-cube}.

When $X\subset \Z^n$ is a digital image with $c_1$-adjacency, it has a unique maximal expression as a union of elementary cubes. For each $q\ge 0$, let $\bar C^{c_1}_q(X)$ be the free abelian group generated by the set of all elementary $q$-cubes in $X$. 

We will give a definition for a boundary operator which differs from the one used in \cite{ke12}, but is rephrased to more closely resemble the boundary operator from the cubical theory.

Define face operators $A_i$ and $B_i$ as follows: for an elementary cube $Q = J_1\times \dots \times J_n$, let:
\begin{align*}
A_iQ &= J_1 \times \dots \times J_{i-1} \times \{ \min J_i \} \times J_{i+1} \times \dots \times J_n, \\
B_iQ &= J_1 \times \dots \times J_{i-1} \times \{ \max J_i \} \times J_{i+1} \times \dots \times J_n.
\end{align*}
Note that when $J_i$ is a degenerate interval, we have $A_iQ = B_iQ$. When $J_i$ is nondegenerate, $A_iQ$ and $B_iQ$ are distinct elementary cubes of dimension one less than the dimension of $Q$.

Now we define the boundary operator: given an elementary $q$-cube $Q = J_1 \times \dots \times J_n$, let $(j_1,\dots,j_q)$ be the sequence of indices for which $J_{j_i}$ is nondegenerate. Then we define:
\[ \pd_q Q = \sum_{i=1}^q (-1)^i (A_{j_i}Q - B_{j_i}Q). \]

It can be verified that $\pd_{q-1} \circ \pd_q = 0$, and thus $(\bar C^{c_1}_q(X), \pd_q)$ is a chain complex. The homology groups are $\bar H^{c_1}_q(X) = \bar Z^{c_1}_q(X) / \bar B_q(X)$.

We immediately have $c_1$ analogues of Theorems \ref{cubicaldim0} and \ref{cubicalisolated}.
\begin{thm}\label{c1dim0}
Let $X\subseteq \Z^n$ be a digital image with $c_1$-adjacency having $d$ components. Then $\bar H^{c_1}_0(X) \cong \Z^d$.
\end{thm}
\begin{proof}
The chain group $\bar C^{c_1}_0(X)$ is generated by the points of $X$, and it is easy to see that two such points are homologous in $\bar H^{c_1}_0(X)$ if and only if there is a path connecting them. (The two points will be the boundary of the chain formed by the path.) Thus $\bar H^{c_1}_0(X)$ has a generator for each component of $X$.
\end{proof}

\begin{thm}\label{c1isolated}
Let $X\subseteq \Z^n$ be a digital image with $c_1$-adjacency which consists of a set of $k$ isolated points. Then $\bar H^{c_1}_0(X) \cong \Z^k$ and $\bar H^{c_1}_q(X) \cong 0$ for $q>0$.
\end{thm}
\begin{proof}
The first part follows immediately from Theorem \ref{c1dim0}. For the second part, observe that if $X$ consists only of isolated points, then the chain group $\bar C^{c_1}_q(X)$ is trivial for all $q>0$. 
\end{proof}

Karaca \& Ege's presentation in \cite{ke12} is different from our $c_1$-cubical theory in some important ways. The theory in \cite{ke12} starts with some digital image $(X,\kappa)$ endowed with some specified cubical structure. Then based on this structure, a homology theory is defined. The work in \cite{ke12} gives no general system for defining a cubical structure on $X$, and thus the resulting homology groups are not intrinsic to the digital image $(X,\kappa)$, but rather depend on the choice of cubical structure. 

For example the digital image $X = [0,1]_\Z^3 \subset \Z^3$ is considered. There is an obvious cubical structure on $X$ as a single elementary $3$-cube together with its faces, and we would expect the cubical homology to be $\Z$ in dimension 0 and trivial in all other dimensions. But instead the calculation in \cite[Theorem 4.5]{ke12} gives $\Z$ in dimensions 0 and 2 and trivial in all other dimensions. This is because the calculation in that theorem (without explicit mention) uses the cubical structure consisting of six $2$-cubes together with their faces, but without the ``solid'' 3-cube. We will see in Example \ref{I3exl} that $\bar H_q^{c_1}(X)$ is indeed $\Z$ in dimension 0, and trivial in other dimensions.

The author has implemented an algorithm with the free mathematical package SageMath to compute the $c_1$-cubical homology groups of any digital image. Source code is available for experimentation at the author's website.\footnote{\url{http://faculty.fairfield.edu/cstaecker}}

The induced homomorphism in this $c_1$-cubical homology theory is nontrivial to develop (there was no effort to define the concept in \cite{ke12}). Given two digital images $X$ and $Y$, both with $c_1$-adjacency, and some continuous $f:X\to Y$, we wish to define an induced homomorphism $\bar f_\#: \bar C_q^{c_1}(X) \to \bar C_q^{c_1}(Y)$ which is a chain map.

Given an elementary $q$-cube $Q \subset X$, we say \emph{$f$ is an embedding on $Q$} if $f(Q)$ is an elementary $q$-cube in $Y$. In this case let $\epsilon_{f,Q} \in \{-1,1\}$ be the orientation with with $f$ maps $Q$ onto $f(Q)$. (This orientation could be defined as the determinant of the affine linear map describing $f$'s restriction to $Q$.) When $f$ is not an embedding on $Q$, we let $\epsilon_{f,Q} = 0$. 

Then we define $\bar f_\#(Q) = \epsilon_{f,Q} (f(Q))$, where the right side is interpreted as the coefficient $\epsilon_{f,Q}$ times the elementary $q$-cube $f(Q) \subset Y$. Extending linearly gives the induced homomorphism $\bar f_\#: \bar C_q^{c_1}(X) \to \bar C_q^{c_1}(Y)$. 

Now to obtain an induced homomorphism in homology, we must show that $\bar f_\#$ is a chain map. We have been unable to prove this analytically, although in low dimensions we can show that $\bar f_\#$ is a chain map by computer enumerations.

\begin{thm}\label{c1chainmap}
Let $X\subset \Z^n$ and $Y\subset \Z^m$ be digital images with $c_1$-adjacency and $n\le 4$, and let $f:X\to Y$ be continuous. Then $f_\#: C^{c_1}_q(X) \to C^{c_1}_q(Y)$ is a chain map.
\end{thm}
\begin{proof}
It suffices to show that $f_\#(\pd Q) = \pd f_\#(Q)$ for any elementary $q$-cube $Q$. We may assume that $q\le n$, since there are no $q$-cubes in $X$ of dimension greater than $n$. For simplicity, by relabeling points (translating to the origin), we may assume that $Q = I^q = [0,1]_\Z^q$, and also that $f(0,\dots,0) = (0,\dots,0)$. Then there are only finitely many possibilities for the behavior of $f$ on $Q$, and we can simply check that $f_\#(\pd Q) = \pd f_\#(Q)$ in all cases. 

We can narrow down the number of cases to check as follows: The set $f(Q)$ is contained in a $q$-dimensional subspace of $\Z^m$, and so it suffices to consider $Y\subset \Z^q$. Also note that $Q$ is a set of diameter $q$, and so since $f$ is continuous, $f(Q)$ has diameter at most $q$. Since $f(Q)$ has diameter $q$ and maps the origin to the origin, we may assume that $f(Q) \subset [-q,\dots,q]_\Z^q$.

Thus the enumeration must only construct all possible continuous functions from $I^q$ to $[-q,\dots,q]_\Z^q$. To further narrow down the search, we assign an ordering to the points of $I_q$ and assume without loss of generality that the first nonzero value of $f$ is the point $(1,0,\dots,0)$. 

With these filters in place, the computation becomes tractable for $q=n \le 4$. When $q=0$ there is only 1 function to check, for $q=1$ there are only 2, for $q=2$ there are 16, and for $q=3$ there are 2128. For $q=4$ there are 23,943,296 functions, requiring 4 days to complete the enumeration, which meets the limit of the author's patience.\end{proof}

Obviously it would be preferable to have a human readable proof of the above, using arguments which suffice in any dimension. We state this as a conjecture.
\begin{conj}\label{chainmapconj}
Let $X\subset \Z^n$ and $Y\subset \Z^m$ be digital images with $c_1$-adjacency, and let $f:X\to Y$ be continuous. Then $f_\#: C_q(X) \to C_q(Y)$ is a chain map.
\end{conj}

\section{Equivalence of simplicial and singular homologies}\label{sxequivsection}

In this section we show that the simplicial and singular homology theories are equivalent. We will follow the argument used in classical algebraic topology to show that the simplicial and singular homology theories of a simplicial complex are equivalent. We will give a complete proof here, following the general idea used in \cite[Theorem 2.27]{hatc02}. 

There is an obvious homomorphism $\alpha:\v C_q(X) \to C_q(X)$ defined by: 
\[ \alpha [x_0,\dots, x_q] = \langle x_0,\dots, x_q\rangle,\] 
and it is easy to check that this is a chain map. Thus we obtain an induced homomorphism $\alpha_*:\v H_q(X) \to H_q(X)$. In this section we show that $\alpha_*$ is an isomorphism for each $q$.

For each $k\ge 0$, let $C_q(X^k)$ be the free abelian group generated by the set of $q$-simplices of dimension $k$ or less. (This group will always either be trivial or equal to $C_q(X)$, so it is not very interesting in its own right, but we introduce the notation in order to make an inductive argument below.) For any $k$, the sequence $(C_q(X^k), \pd)$ is a chain complex, a subcomplex of $(C_q(X),\pd)$, and we will have $C_q(X^k) \le C_q(X^{k+1}) \le C_q(X)$, where $\le$ indicates a subgroup. 

There are natural inclusion and projection maps which make the following sequence exact:
\[ 0 \to C_q(X^{k-1}) \to C_q(X^k) \to C_q(X^k)/C_q(X^{k-1})  \to 0, \]
and so we may construct a long exact relative homology sequence as in \eqref{longexact}:

\[\dots \to H_{q+1}(X^k ,X^{k-1}) \to H_q(X^{k-1}) \to H_q(X^k) \to H_q(X^k, X^{k-1}) \to H_{q-1}(X^{k-1})\to \dots \]

Similarly for each $k\ge 0$, let $\v C_q(X^k)$ be the free abelian group generated by the set of singular $q$-simplices $\phi$ such that $\phi(\Delta^q)$ is a simplex of dimension $k$ or less. The group $\v C_q(X^k)$ is trivial when $q<k$ but nontrivial for $q \ge k$. 

For any $k$, the sequence $(\v C_q(X^k), \pd)$ is a chain complex, a subcomplex of $(\v C_q(X),\pd)$, and we will have $\v C_q(X^k) \le \v C_q(X^{k+1}) \le \v C_q(X)$. 

Again from \eqref{longexact} we have a long exact relative homology sequence:
\[\dots \to \v H_{q+1}(X^k ,X^{k-1}) \to \v H_q(X^{k-1}) \to \v H_q(X^k) \to \v H_q(X^k, X^{k-1}) \to \v H_{q-1}(X^{k-1})\to \dots \]

The same map $\alpha$ above will restrict to a chain map of $\v C_q(X^k) \to C_q(X^k)$ and induce a chain map of $\v C_q(X^k,X^{k-1}) \to C_q(X^k,X^{k-1})$, and thus we obtain the following commutative diagram using the relative homology sequence, where the vertical arrows are induced by $\alpha$.
\begin{equation}\label{5lemeq}
\begin{tikzcd}
\v H_{q+1}(X^k,X^{k-1}) \arrow[r] \arrow[d] & \v H_q(X^{k-1}) \arrow[r] \arrow[d] & \v H_q(X^k) \arrow[r] \arrow[d] & \v H_{q}(X^k, X^{k-1}) \arrow[d]\arrow[r] & \v H_{q-1}(X^{k-1}) \arrow[d] \\
H_{q+1}(X^k ,X^{k-1}) \arrow[r] & H_q(X^{k-1}) \arrow[r] & H_q(X^k) \arrow[r] &H_q(X^k, X^{k-1}) \arrow[r] & H_{q-1}(X^{k-1}) 
\end{tikzcd}
\end{equation}

\begin{lem}\label{relativelem}
For each $q$, the function $\bar \alpha_q: \v H_q(X^k, X^{k-1}) \to H_q(X^k,X^{k-1})$ induced by $\alpha$ is an isomorphism.
\end{lem}
\begin{proof}
It is clear from its definition that $\alpha$ is surjective, so $\bar \alpha$ will be surjective, and it is enough to show that $\bar \alpha$ is injective. 

The chain group $\v C_q(X^k, X^{k-1}) = \v C_q(X^k) / \v C_q(X^{k-1})$ is the free abelian group generated by all singular $q$-simplices $\phi$ of $X^k$ such that $\phi(\Delta^q)$ is a simplex of dimension $k$. 
When $q<k$ this group is trivial, and so $\v H_q(X^k,X^{k-1})$ is trivial, and thus $\bar \alpha$ is injective for $q<k$.

It remains to consider the case $q\ge k$. Let $\bar \alpha_q:\v C_q(X^k,X^{k-1}) \to C_q(X^k,X^{k-1})$ be the homomorphism induced by $\alpha$. To show that $\ker \bar \alpha_{*,q}$ is trivial, it suffices to show that $\ker \bar \alpha_{q} \subset \v B_q(X^k,X^{k-1})$. By the definition of $\alpha$, we see that $\ker \bar \alpha_{q}$ is generated by all sums $\phi + \psi$ where $\phi,\psi \in \v C_q(X^k,X^{k-1})$ and $\psi$ is an odd permutation of $\phi$. Any odd permutation can be expressed as a composition of an odd number of transpositions, and so to show that $\ker \alpha_{q} \subset \v B_q(X^k,X^{k-1})$, it suffices to show $\phi + \psi \in \v B_q(X^k,X^{k-1})$ whenever $\psi$ is a transposition of $\phi$. That is, given $\phi = [x_0,\dots,x_q]$ and any $i \in \{0,\dots,q-1\}$, we must show that:
\begin{equation}\label{inBk}
[x_0,\dots,x_i,x_{i+1},\dots,x_q] + [x_0, \dots, x_{i+1}, x_i, \dots, x_q] \in \v B_k(X^q,X^{q-1}). 
\end{equation}

The assumption that $\phi \in \v C_q(X^k,X^{k-1})$ means that the set $\{x_0,\dots,x_q\}$ consists of exactly $k$ distinct points. Let $\sigma = [x_0,\dots,x_{i+1},x_i,x_{i+1},\dots, x_q] \in C_{q+1}(X^k,X^{k-1})$, and we compute:
\[ \begin{split}
\pd_{q+1}\sigma =& [x_1,\dots,x_{i+1},x_i,x_{i+1},\dots, x_q] + \dots + (-1)^i [x_0,\dots,x_i,x_{i+1},\dots,x_q] \\
& + (-1)^{i+1} [x_0,\dots,x_{i+1},x_{i+1},\dots,x_q] + (-1)^{i+2}[x_0,\dots,x_{i+1},x_i,\dots,x_q] \\
& + (-1)^{q+1} [x_0,\dots,x_{i+1},x_i,x_{i+1},\dots, x_{q-1}]
\end{split} \]
Any term above listing fewer than $k$ distinct points will be zero in $\v C_q(X^k,X^{k-1})$, and so most of the terms above are zero. We are left with:
\[ 
\pd_{q+1}\sigma = (-1)^i ([x_0,\dots,x_i,x_{i+1},\dots,x_q] + [x_0,\dots,x_{i+1},x_i,\dots,x_q])
\]
which establishes \eqref{inBk}.
\end{proof}

Our main theorem requires a very weak finiteness condition on $X$. Any finite digital image will satisfy the condition, as will any image $X\subset \Z^n$ with $c_k$-adjacency for any $k$. The main work of the proof has already been accomplished by Lemma \ref{relativelem}, the proof below is simply a homological argument.
\begin{thm}\label{simplicialiso}
Let $(X,\kappa)$ be a digital image, and assume there is some dimension $k$ for which $X$ contains no simplicies of dimension greater than $k$. Then for each $q$, we have $H_q(X) \cong \v H_q(X)$. 
\end{thm}
\begin{proof}
We prove the theorem by induction on $q$ and $k$.
For $q=0$, we have $H_0(X) \cong \v H_0(X)$, since by Theorems \ref{simplicialdim0} and \ref{singulardim0} these homology groups are both $\Z^d$ where $d$ is the number of components of $X$.

When $k=0$ then $X$ simply consists of a set of isolated points, and so $H_q(X) \cong \v H_q(X)$ by Theorems \ref{simplicialisolated} and \ref{singularisolated}.

Now we consider the inductive case. Since $X$ contains only simplices of dimension $k$ or less, we will have $C_q(X) = C_q(X^k)$ and $\v C_q(X) = \v C_q(X^k)$ for all $q$. 
Thus $H_q(X) = H_q(X^k)$ and $\v H_q(X) = \v H_q(X^k)$ for all $q$, and it suffices to show that the vertical arrow in the middle of \eqref{5lemeq} is an isomorphism. 

By the Five Lemma (see \cite{hatc02}), it suffices to show that the other 4 vertical arrows are isomorphisms. By induction in $k$, the second vertical arrow $\v H_q(X^{k-1}) \to H_q(X^{k-1})$ is an isomorphism, and by induction in $k$ and $q$ the last vertical arrow $\v H_{q-1}(X^{k-1}) \to H_{q-1}(X^{k-1})$ is an isomorphism. By Lemma \ref{relativelem}, the first and fourth arrows are isomorphisms. 
\end{proof}

The author's original intention was to additionally prove a cubical version of Theorem \ref{simplicialiso}, that is, if $X\subset \Z^n$ is a digital image with $c_1$-adjacency, then $\bar H_q(X) \cong \bar H^{c_1}_q(X)$ for each $q$. 

This turned out to be much harder than anticipated: even constructing a chain map $\beta:\bar C_q(X) \to \bar C^{c_1}_q(X)$ is difficult. We believe it should be possible, but we will simply state it as a conjecture:

\begin{conj}\label{cubicalconjecture}
Let $X \subset \Z^n$ be a digital image with $c_1$-adjacency. Then $\bar H_q(X)$ and $\bar H_q^{c_1}(X)$ are isomorphic.
\end{conj}

It seems likely that the conjecture could be verified in low dimensions by computer enumerations, but we have not attempted this.

\section{Homotopy invariance in simplicial and cubical homology}\label{homologysec}
In this section we discuss the homotopy invariance of the induced homomorphism on the various homology groups. In \cite{ja19} it is shown that, when $f$ and $g$ are homotopic, their induced maps on cubical homology groups are the same.

We will prove the same property holds for $c_1$-cubical homology, and that a similar result holds for simplicial homology, but in this case the homotopy must be strong.

\begin{thm}\label{prism}
If $X$ is finite and $f,g:X\to Y$ are strongly homotopic, then the induced homomorphisms $f_*,g_*:H_q(X) \to H_q(Y)$ are equal for each $q$.
\end{thm}
\begin{proof}
By induction and Theorem \ref{punctuatedthm}, it suffices to prove the result when $f$ and $g$ are homotopic by a punctuated homotopy in one step. 
We mimic the proof for this result in classical homology theory, see for example Proposition 2.10 of \cite{hatc02}. For each $q$, we define the ``prism operator'' $P:C_q(X) \to C_{q+1}(Y)$ as follows: For $\sigma\in C_q(X)$ with $\sigma = \langle x_0,\dots,x_q\rangle$, let:
\[ P(\sigma) = \sum_{j=0}^q (-1)^j \langle f(x_0),\dots, f(x_j), g(x_j),\dots, g(x_q)\rangle, \]
where the term $\langle f(x_0),\dots, f(x_j), g(x_j),\dots, g(x_q)\rangle$ is interpreted as $0$ if any of these points are equal.

Note that the definition of $P$ only makes sense if $\{f(x_0),\dots, f(x_j), g(x_j),\dots, g(x_q)\}$ is indeed a set of $q+2$ points that are mutually adjacent or equal. This is ensured because $f$ and $g$ are strongly homotopic in 1 step, and thus by Theorem \ref{stronghtpadjs}, since $\{x_0,\dots, x_q\}$ are mutually adjacent, the points of $\{f(x_0),\dots, f(x_q),g(x_0),\dots, g(x_q)\}$ are mutually adjacent or equal. This is the point at which the proof will fail if the homotopy is not strong.

The bulk of the proof consists of proving the following formula:
\begin{equation}\label{prismeq} 
\pd (P(\sigma)) = g_\#(\sigma) - f_\#(\sigma) - P(\pd \sigma). 
\end{equation}

Since $f$ and $g$ are homotopic in one step by punctuated homotopy, there is some $x'\in X$ such that $f(x)=g(x)$ for all $x\neq x'$. If $\sigma = \langle x_0,\dots,x_q\rangle$ with $x_i \neq x'$ for all $i$, then $f(x_i)=g(x_i)$ for each $i$. Thus $P(\sigma) = 0$, since $\langle f(x_0),\dots, f(x_j), g(x_j),\dots, g(x_q)\rangle$ will repeat the point $f(x_j)=g(x_j)$. Thus, whenever $\sigma$ does not use $x'$, we have $P(\sigma)=0$.

Formula \eqref{prismeq} is easy to prove when $\sigma$ does not use the vertex $x'$: in that case $P(\sigma)=0$ and thus the left side of \eqref{prismeq} is 0. For the right side, note that $\pd \sigma$ also does not use the point $x'$, and so we have $P(\pd \sigma) = 0$. Also since $\sigma$ does not use $x'$, we will have $f_\#(\sigma) = g_\#(\sigma)$, and thus the right side of \eqref{prismeq} is also 0, and we have proved \eqref{prismeq}.

Now we prove \eqref{prismeq} in the case when $\sigma$ does use the point $x'$. Without loss of generality assume $\sigma = \langle x', x_1,\dots,x_q\rangle$. In this case we have
\begin{align*} 
P(\sigma) &= \langle f(x'),g(x'),g(x_1),\dots, g(x_q) \rangle \\
&\qquad + \sum_{j=1}^q (-1)^j \langle f(x'),f(x_1),\dots, f(x_j), g(x_j), \dots, g(x_q) \rangle \\
&= \langle f(x'),g(x'),g(x_1),\dots, g(x_q) \rangle
\end{align*}
where most of the terms above are 0 because they repeat the point $f(x_j)=g(x_j)$. 
Then the left side of \eqref{prismeq} is
\begin{align*}
\pd (P(\sigma)) &= \pd(\langle f(x'),g(x'), g(x_1), \dots, g(x_q)\rangle) \\
&= \langle g(x'), g(x_1), \dots, g(x_q)\rangle - \langle f(x'), g(x_1), \dots, g(x_q)\rangle \\
&\qquad+ \sum_{i=1}^q (-1)^{i+1} \langle f(x'),g(x'),g(x_1),\dots, \widehat {g(x_i)}, \dots, g(x_q) \rangle 
\end{align*}
Since $g(x_i)=f(x_i)$, the above simplifies to:
\[ \pd (P(\sigma)) = g_\#(\sigma) - f_\#(\sigma) + \sum_{i=1}^q (-1)^{i+1} \langle f(x'),g(x'),g(x_1),\dots, \widehat {g(x_i)}, \dots, g(x_q) \rangle. \]

To prove \eqref{prismeq}, it suffices to show that the summation above equals $-P(\pd(\sigma))$. We have:
\begin{align*}
P(\pd(\sigma)) &= P\left( \langle x_1,\dots,x_q\rangle + \sum_{i=1}^q (-1)^i \langle x', x_1,\dots, \widehat{x_i}, \dots, x_q\rangle \right) \\
&= P\left(\sum_{i=1}^q (-1)^i \langle x', x_1,\dots, \widehat{x_i}, \dots, x_q\rangle \right) 
\end{align*}
where $P(\langle x_1,\dots,x_q\rangle) = 0$ since this simplex does not use $x'$. Now when we apply $P$ above, the only nonzero terms are those with $j=0$ in the definition of $P$. All others will repeat some point $f(x_i)=g(x_i)$. Thus we have:
\begin{align*}
P(\pd(\sigma)) &= \sum_{i=1}^q (-1)^i \langle f(x'),g(x'),g(x_1),\dots,\widehat{g(x_i)}, \dots, g(x_n) \rangle \\
&= -  \sum_{i=1}^q (-1)^{i+1} \langle f(x'),g(x'),g(x_1),\dots,\widehat{g(x_i)}, \dots, g(x_n) \rangle 
\end{align*}
and we have proved \eqref{prismeq}. 

The formula \eqref{prismeq} holds when $\sigma$ is any simplex, and so by linearity it will hold for any chain. Now let $\alpha\in Z_q(X)$ be a $q$-cycle, so $\pd(\alpha) = 0$ and thus $P(\pd(\alpha))=0$. Then by \eqref{prismeq} we have:
\[ g_\#(\alpha) - f_\#(\alpha) = \pd P(\alpha) \in B_q(Y). \]
Thus $g_\#(\alpha)$ and $f_\#(\alpha)$ differ by a $q$-boundary, that is, 
$g_*(\alpha) = f_*(\alpha) \in H_q(Y)$.
\end{proof}

Now we prove a homotopy invariance property for $c_1$-cubical homology. Since the result concerns the induced homomorphism in $c_1$-cubical homology, we must require that $f_\#:\bar C_q^{c_1}(X) \to \bar C_q^{c_1}(Y)$ is a chain map. Barring a proof of Conjecture \ref{chainmapconj}, we can only demonstrate the theorem in low dimensions.

\begin{thm}\label{c1prism}
Let $X\subseteq \Z^m$ and $Y\subseteq \Z^n$ be digital images both with $c_1$-adjacency, with $m\le 3$, and let $f,g:X\to Y$ be homotopic. Then the induced homomorphisms $\bar f_*,\bar g_*:\bar H_q^{c_1} (X) \to \bar H_q^{c_1}(Y)$ are equal for each $q$. If Conjecture \ref{chainmapconj} is true, then this holds for any $m$. 
\end{thm}
\begin{proof}
As in the proof of Theorem \ref{prism}, we may assume that $f$ is homotopic to $g$ in 1 step, that is, that $f(x) \adjeq g(x)$ for each $x$. By the same homological argument at the end of the proof of Theorem \ref{prism}, it will suffice to define a homomorphism $P:\bar C^{c_1}_q(X) \to \bar C^{c_1}_{q+1}(Y)$ which satisfies:
\begin{equation}\label{c1prismeq}
\pd P(Q) = \bar g_\#(Q) - \bar f_\#(Q) - P(\pd Q).
\end{equation}

For an elementary $q$-cube $Q\subset X$, we define $P(Q) \in \bar C^{c_1}_{q+1}(Y)$ as follows: if the set $f(Q) \cup g(Q)$ is an elementary $(q+1)$-cube, then $P(Q) = f(Q) \cup g(Q) \in \bar C^{c_1}_{q+1}(Y)$. Otherwise we define $P(Q) = 0$. Extending the definition linearly defines a homomorphism $P:\bar C^{c_1}_q(X) \to \bar C^{c_1}_{q+1}(Y)$. 

Let $T:I\times X \to Y$ be defined as $T(0,x) = f(x)$ and $T(1,x) = g(x)$. Since $f(x)\adjeq g(x)$ for each $x$, this function $T$ is continuous (using $\NP_1$ adjacency in the product $I\times X$). The induced homomorphism $\bar T_\#:C_{q+1}(I\times X) \to C_{q+1}(Y)$ is closely related to $P$. In particular $P(Q) = \bar T_\#(I\times Q)$, and when $i >1$ we have:
\[ \bar T_\#(A_i(I \times Q)) = P(A_{i-1} Q), \qquad \bar T_\#(B_i(I\times Q)) = P(B_{i-1} Q). \]
When $i=1$, we instead have $\bar T_\#(A_1(I\times Q)) = \bar T_\#(\{0\}\times Q) = f_\#(Q)$ and similarly $\bar T_\#(B_{1}(I\times Q) = g_\#(Q)$.

Since $X \subset \Z^3$ we may consider $I \times X \subset \Z^4$, and so by Theorem \ref{c1chainmap} we will have $\bar T_\#(\pd \sigma) = \pd \bar T_\#(\sigma)$ for any chain $\sigma$. Thus we have:
\begin{align*}
\pd P(Q) &= \pd \bar T_\#(I\times Q) = \bar T_\#(\pd (I \times Q))  \\
&= \sum_{i=1}^{q+1} (-1)^i (T_\#(A_i(I\times Q)) - T_\#(B_i(I\times Q))) \\
&= (g_\#(Q) - f_\#(Q)) + \left( \sum_{i=2}^{q+1} (-1)^i (P(A_{i-1}Q) - P(B_{i-1}Q)) \right) \\
&= g_\#(Q) - f_\#(Q) - P(\pd Q)
\end{align*}
which establishes \eqref{c1prismeq}.
\end{proof}

Theorem \ref{c1prism} will imply that, if $X$ and $Y$ are homotopy equivalent, then they have the same $c_1$-cubical homology groups. We immediately obtain:
\begin{cor}\label{c1contractible}
If $X \subset \Z^n$ is a contractible digital image with $c_1$-adjacency, and assume $n\le 3$ or that Conjecture \ref{chainmapconj} is true. Then:
\[ \bar H_q^{c_1}(X) = \begin{cases} \Z &\text{ if } q=0, \\
0 &\text{ if } q>0. \end{cases} \]
\end{cor}

\section{Relationships between cubical and simplicial homology}\label{relationshipsection}
So far we have exhibited four homology theories: simplicial, singular, cubical, and $c_1$-cubical. The simplicial and singular theories are isomorphic, and we have conjectured that the two cubical theories are isomorphic. Thus, subject to the conjecture, there are two different types of homology under discussion: simplicial and cubical. The simplicial and cubical theories are not isomorphic, as we will see in the next section. In this section we consider some relationships that do exist between the simplicial and cubical theories.

We have already seen results that imply that the simplicial and cubical theories are the same in dimension 0. Theorems \ref{simplicialdim0}, \ref{cubicaldim0}, and \ref{c1dim0} give:
\begin{thm}\label{H0thm}
Let $X$ be any digital image. Then:
\[ H_0(X) = \bar H_0(X) = \Z^d \]
where $d > 0$ is the number of connected components of $X$. If $X\subset \Z^n$ with $c_1$-adjacency, then also $\bar H_0^{c_1}(X) = \Z^d$.
\end{thm}

In dimension 1, there is also a relationship between simplicial and cubical homology. 
The chain group $\v C_1(X)$ is generated by the set of all maps $\phi: \Delta^1 \to X$, while $\bar C_1(X)$ is generated by the set of all maps $\sigma: I^1 \to X$. But $\Delta^1$ and $I^1$ are the same digital image-- so we may identify chain groups $\v C_1(X) = \bar C_1(X)$, and the singular and cubical boundary maps $\pd_1$ are the same. Thus we may also identify the groups of cycles $\v Z_1(X) = \bar Z_1(X)$.

This identification induces a natural homomorphism $\v H_1(X) \to \bar H_1(X)$ which simply regards a singular $1$-cycle as a cubical $1$-cycle.

In exactly the same way, when $X\subset \Z^n$ with $c_1$-adjacency there is a natural homomorphism $H_1(X) \to \bar H_1^{c_1}(X)$ induced by the chain map which regards each $1$-simplex as an elementary $1$-cube.

\begin{thm}\label{H1surjection}
For any digital image $X$, the natural map $\v H_1(X) \to \bar H_1(X)$ is surjective. When $X\subset \Z^n$ with $c_1$-adjacency then the map $H_1(X) \to \bar H_1^{c_1}(X)$ is surjective.
\end{thm}
\begin{proof}
We have $\v H_1(X) = \v Z_1(X) / \v B_1(X)$ and $\bar H_1(X) = \bar Z_1(X) / \bar B_1(X)$. Thus it suffices to show that, when we identify $\v Z_1(X) = \bar Z_1(X)$, we have $\v B_1(X) \le \bar B_1(X)$. 

For clarity, write the singular and cubical boundary operators as $\v \pd_2:\v C_2(X) \to \v C_1(X)$ and $\bar \pd_1:\bar C_2(X) \to \bar C_1(X)$. When $\phi = [x_0,x_1,x_2]$ is a singular $2$-simplex, then define $L(\phi) = \sigma$ to be the following $2$-cube:
\begin{align*}
\sigma(0,0) &= x_0 & \sigma(0,1) &= x_1 \\
\sigma(1,0) &= x_0 & \sigma(1,1) &= x_2
\end{align*}
and $L$ extends to a map $L:\v C_2(X) \to \bar C_2(X)$. It is routine to check that $\v \pd_2(\phi) = \bar \pd_2(L(\phi))$ for any $\phi \in \v C_2(X)$ when we identify $\v C_1(X) = \bar C_1(X)$. Thus we will have $\v B_1(X) \le \bar B_1(X)$ as desired.

Now when $X\subset \Z^n$ with $c_1$-adjacency, similarly we must show that $B_1(X) \le \bar B_1^{c_1}(X)$ but this is obvious because $C_2(X)$ is trivial since a digital image with $c_1$-adjacency contains no $2$-simplex. Thus $B_1(X) = \pd (C_2(X))$ is trivial and so is automatically a subgroup.
\end{proof}

Since $\v H_1(X) \cong H_1(X)$, the above implies that there is always a surjection $H_1(X) \to \bar H_1(X)$. 
We will see in the next section that $H_1(X)$ need not be isomorphic to $\bar H_1(X)$, and that there need not be any surjection $H_q(X) \to \bar H_q(X)$ when $q>1$. 

\section{Examples}\label{exlsection}
In this section we provide some examples comparing the simplicial and cubical homology groups of various images. Most of the examples for simplicial homology appear already in the literature, but no examples for $c_1$-cubical homology have been computed.

The simplest examples are the simplicial and cubical homology groups for single points. We have already seen this result as Theorems \ref{simplicialisolated}, \ref{cubicalisolated}, and \ref{c1isolated}.
\begin{thm}
Let $X$ be an image consisting of $d$ isolated points. Then:
\[ H_q(X) = \bar H_q(X) = \bar H_q^{c_1}(X) = \begin{cases} \Z^d  &\text{ if } q=0, \\
0 &\text{ if } q>0. \end{cases} \]
\end{thm}

Now we describe some more interesting examples. A fruitful source of example digital images is the \emph{digital cycle} $C_m$, the digital image consisting of $n$ points $x_1,\dots,x_m$ where $x_i \adj x_j$ only when $j=i\pm 1$, where we read $i$ and $j$ modulo $n$. As we will see, the simplicial and cubical homologies for digital cycles (as far as we can compute them) agree in all cases except $H_1(C_4) = \Z \neq 0 = \bar H_1^{c_1}(C_4)$. 

We will use a lemma which is straightforward but interesting. For two digital images $(X,\kappa)$ and $(Y,\lambda)$, we say $X$ \emph{embeds} in $Y$ if $X$ is $(\kappa,\lambda)$-isomorphic to a subset of $Y$.

\begin{lem}
The digital cycle $C_m$ embeds in $(\Z^n,c_1)$ for some $n$ if and only if $m$ is even.
\end{lem}
\begin{proof}
When we view $(\Z^n,c_1)$ as a graph, it is bipartite, with the two parts given by:
\begin{align*} S^+ &= \{ (x_1,\dots,x_n) \mid x_1 + \dots + x_n \text{ is even} \}, \\
S^- &= \{ (x_1,\dots,x_n) \mid x_1 + \dots + x_n \text{ is odd} \}. 
\end{align*}
Thus $C_m$ can only embed into $(\Z^n,c_1)$ if it too is bipartite, and this is only possible when $m$ is even. We have shown that if $m$ embeds in $(\Z^n,c_1)$, then $m$ is even.

For the converse, let $m>0$ be even and we will exhibit a subset of $(\Z^n,c_1)$ for some $n$ isomorphic to $C_m$. When $m=2$, then we may simply use $[0,1]_\Z \subset \Z^1$, which is isomorphic to $C_2$. When $m=4$, then we may use $[0,1]^2 \subset \Z^2$, which is isomorphic to $C_4$. When $m=6$, we observe that $C_6$ is isomorphic to the following subset of $\Z^3$, using $c_1$-adjacency: 
\[ \{(0,0,0),(1,0,0),(1,1,0),(1,1,1),(0,1,1),(0,0,1) \} \]

Finally when $m>6$ is even, the cycle $C_m$ is isomorphic to the following subset of $\Z^2$:
\[ ([0,m/2-1] \times \{0,2\}) \cup \{ (0,1), (m/2 -1,1) \}. \qedhere \]
\end{proof} 

Now we can compute the various homology groups of $C_m$. Whenever possible (when $m$ is even), we will consider $C_m$ as embedded in $\Z^n$ with $c_1$-adjacency.

\begin{thm}\label{cyclecomputation}
For $m < 4$, we have:
\[ H_q(C_m) = \bar H_q(C_m) = \bar H_q^{c_1}(C_2) =\begin{cases} \Z &\text{ if } q=0, \\
0 &\text{ if } q>0. \end{cases} \]
For $m=4$, we have:
\[ H_q(C_4) = \begin{cases} \Z &\text{ if } q=0, \\
\Z &\text{ if } q=1, \\
0 &\text{ if } q>1, \end{cases}
\qquad
\bar H_q(C_4) = \bar H_q^{c_1}(C_4) = \begin{cases} \Z &\text{ if } q=0, \\
0 &\text{ if } q>0. \end{cases}
\]
For $m > 4$ we have:
\[  H_q(C_m) = \begin{cases} \Z &\text{ if } q=0, \\
\Z &\text{ if } q=1, \\
0 &\text{ if } q>1, \end{cases} 
\qquad
\bar H_q(C_m) = \begin{cases} \Z &\text{ if } q=0, \\
\Z &\text{ if } q=1. \end{cases}
\]
When $m>4$ is even, we have $H_q(C_m) = \bar H_q^{c_1}(C_m)$.
\end{thm}
\begin{proof}
The groups $H_q(C_m)$ were already computed fully in Theorem \ref{Cnhomology}, so we need only prove the statements concerning $\bar H_q(C_m)$ and $\bar H^{c_1}_q(C_m)$. 

For $\bar H_q(C_m)$, we see that $\bar H_0(C_m) = \Z$ by Theorem \ref{H0thm}, and $\bar H_q(C_4)=0$ for $q>0$ by Corollary \ref{contractiblecube} since $C_4$ is contractible. In fact $C_m$ is contractible for all $m\le 4$, and so we have $\bar H_q(C_m) = 0$ for $q>1$ and $m\le 4$. For $m>4$ we have $\bar H_1(C_m) \cong \Z$ by the Hurewicz Theorem of \cite{ja19}, since $\pi_1(C_m) = \Z$ for $m>4$. 

For $\bar H_q^{c_1}(C_2)$, the required statement is clear because $C_2$ is connected with no elementary cubical $q$-cycles for any $q>0$. 

For $\bar H_q^{c_1}(C_4)$, the statement for $q=0$ follows from Theorem \ref{H0thm}. For $q=1$ the group of elementary cubical $1$-cycles has a single generator, but it is the boundary of an elementary 2-cube, and so $\bar H^{c_1}_1(C_4) = 0$. For $q>1$ there are no elementary $q$-cycles, and so $\bar H^{c_1}_q(C_4) = 0$ when $q>1$.

It remains to compute $\bar H_q^{c_1}(C_m)$ when $m>4$ is even. Since $C_m$ is connected we have $\bar H^{c_1}_0(C_m) = \Z$ by Theorem \ref{H0thm}. The group of elementary $1$-cycles has a single generator which is not the boundary of any elementary cubical $2$-chain, and so $\bar H^{c_1}_1(C_m) = \Z$. For $q>1$ there are no elementary cubical $q$-cycles, and so $\bar H_q^{c_1}(C_m) = 0$ as desired.
\end{proof}

It is obviously to be expected that
\[
H_q(C_m) = \bar H_q(C_m) = \begin{cases} \Z &\text{ if } q=0, \\
\Z &\text{ if } q=1, \\
0 &\text{ if } q>1, \end{cases} 
\]
for all values $m>4$ including odd values, though this seems difficult to prove. The definitions of cubical homology make even the group $\bar H_2(C_5)$ very hard to compute by hand. So we will state this as a conjecture:
\begin{conj}
For $m>4$, we have:
\[
H_q(C_m) = \bar H_q(C_m) = \begin{cases} \Z &\text{ if } q=0, \\
\Z &\text{ if } q=1, \\
0 &\text{ if } q>1, \end{cases} 
\]
\end{conj}


One example where the simplicial and cubical homologies are quite different is the standard 3-cube $I^3 = [0,1]_\Z^3$ taken with $c_1$-adjacency.
\begin{exl}\label{I3exl}
Consider the digital image $I^3$ with $c_1$-adjacency. We have:
\[ H_q(I^3) = \begin{cases} \Z &\text{ if } q = 0, \\
\Z^5 &\text{ if } q=1, \\
\Z & \text{ if } q=2,\\
0 &\text{ if } q>2, \end{cases} \qquad \bar H_q(I^3) = \bar H_q^{c_1}(I^3) = \begin{cases} \Z  &\text{ if } q=0, \\
0 &\text{ if } q>0. \end{cases} \]
\end{exl}
\begin{proof}
The computation of the simplicial homology was done in Theorem 3.20 of \cite{bko11}, where the image in question is called $\mathit{MSS}'_6$. The appearance of $\Z^5$ as the homology group in dimension 1 is surprising. The generators of $H_1(I^3)$ are formed by making $1$-cycles around each of the 6 faces of the cube. This produces six $1$-cycles which are not boundaries, but they are not linearly independent-- each one can be obtained by a combination of the other 5, and thus we have only 5 linearly independent $1$-cycles. The details of the computation are given in \cite{bko11}.

In the case of cubical homology, each of these $1$-cycles is indeed a boundary of the $2$-cube which makes the corresponding face of the cube. Indeed $I^3$ is contractible, and so the computations of $\bar H_q(I^3)$ and $\bar H_q^{c_1}(I^3)$ follow from Corollary \ref{contractiblecube} and \ref{c1contractible}.
\end{proof}

\begin{exl}
Let $X = [0,2]_\Z^3 - \{(1,1,1)\}$, taken with $c_1$-adjacency. This digital image is called $\mathit{MSS}_6$ in \cite{bko11}, though its homology is not computed, presumably because $H_1(X)$ is very large. By Theorem \ref{H0thm} we have $H_0(X) = \bar H^{c_1}_0(X) = \Z$. 

In dimension 1 there are 24 different 1-cycles tracing around unit squares. In $\bar H^{c_1}_1(X)$ these are all trivial since these cycles are boundaries of $2$-cubes. In $H_1(X)$ these are all nontrivial, although as in $I^3$ one of these cycles can be written as a sum of the other 23. The computer implementations confirm that $H_1(X) \cong \Z^{23}$ and $\bar H_1^{c_1}(X) \cong 0$. 

In dimension 2 there are no 2-simplices, so $H_2(X)=0$. The above mentioned 24 unit squares, when taken together, form a $2$-cycle in $\bar Z^{c_1}_2(X)$ which is not the boundary of any $3$-cycle, and thus $\bar H_2^{c_1}(X)\neq 0$. This example is tractable for our computer implementation, which confirms that $\bar H_2^{c_1}(X) = \Z$. 

For $q>2$ there are no $q$-simplices or elementary $q$-cubes, so $H_q(X) = \bar H^{c_1}_q(X) = 0$. In summary, we have:
\[ H_q(X) = \begin{cases} \Z &\text{ if } q=0 \\
\Z^{23} &\text{ if } q=1 \\
0 &\text{ if } q>1\end{cases}
\quad
\bar H_q^{c_1}(X) = \begin{cases}
\Z &\text{ if } q=0 \\
0 &\text{ if } q=1 \\
\Z &\text{ if } q=2\\
0 &\text{ if } q>2\end{cases} \]

Note that $H_2(X) = 0$, while $\bar H^{c_1}_2(X) \cong \Z$. Thus the example demonstrates that, in contrast with Theorem \ref{H1surjection}, there is not always a surjective homomorphism of $H_q(X) \to \bar H^{c_1}_q(X)$ when $q>1$.
\end{exl}

\bibliographystyle{hplain}

\end{document}